\documentclass[]{article}      
\usepackage{arxiv}
\usepackage{macros_thesis}
\usepackage{amsmath, amsthm}
\usepackage{hyperref} 
\usepackage{cleveref}
\usepackage{physics}
\usepackage{todonotes}

\usepackage{mathdots}
\usepackage{siunitx}
\usepackage{url}
\usepackage{amssymb}
\usepackage{latexsym}
\theoremstyle{definition}
\newtheorem{thm}{Theorem}[section]

\newtheorem{mydef}{Definition}[section]
\newtheorem{rmk}{Remark}[section]

\newtheorem{prop}{Proposition}[section]
\newtheorem{lem}{Lemma}[section]

\title{A pair of second-order, complex-valued, $\Nop$-split
		operator-splitting methods
		\thanks{This work was funded by the Natural Sciences and
				Engineering Research Council of Canada through its Discovery
				Grant Program [RGPN-2020-04467] (RJS).}}

\author{ \href{}{\hspace{1mm}Raymond J. Spiteri}\\
	Department of Computer Science\\
	University of Saskatchewan, Saskatoon, SK, Canada\\
	\\
	\texttt{spiteri@cs.usask.ca} \\
	\And
	\href{}{\hspace{1mm}Siqi Wei}\\
	Department of Mathematics and Statistics \\ 
	University of Saskatchewan, Saskatoon, SK, Canada \\\\
	\texttt{siqi.wei@usask.ca} \\
}

\date{}

\renewcommand{\shorttitle}{Complex-valued $\Nop$-split OS}

\begin{document}
\maketitle

\begin{abstract}
  The use of operator-splitting methods to solve differential
  equations is widespread, but the methods are generally only defined
  for a given number of operators, most commonly two. Most
  operator-splitting methods are not generalizable to problems with
  $\Nop$ operators for arbitrary~$\Nop$. In fact, there are only two
  known methods that can be applied to general $\Nop$-split problems:
  the first-order Lie--Trotter (or Godunov) method and the
  second-order Strang (or Strang--Marchuk) method. In this paper, we
  derive two second-order operator-splitting methods that also
  generalize to $\Nop$-split problems. These methods are complex
  valued but have positive real parts, giving them favorable stability
  properties, and require few sub-integrations per stage, making them
  computationally inexpensive. They can also be used as base methods
  from which to construct higher-order $\Nop$-split operator-splitting
  methods with positive real parts. We verify the orders of accuracy
  of these new $\Nop$-split methods and demonstrate their favorable
  efficiency properties against well-known real-valued
  operator-splitting methods on both real-valued and complex-valued
  differential equations.
\end{abstract}

%




\section{Introduction}
\label{complexLT_intro}

Operator-splitting methods are widely used to solve some of the most
challenging differential equations that arise in science and
engineering. Such methods offer a divide-and-conquer approach to solve
problems that may otherwise be unsolvable monolithically due to
prohibitive memory or wall-clock time requirements or because the
problems themselves are solved by separate software libraries as in
\emph{co-simulation}~\cite{Gomes2018-coSimulation}.

Applications of operator-splitting methods often split the
differential equation into two operators, e.g., reaction/diffusion,
linear/nonlinear, stiff/non-stiff, etc. However, many multi-physics
problems can often be naturally divided into more than two operators,
e.g., advection/reaction/diffusion, multiple scales, multiple
dimensions, etc., that could benefit from the use of specialized
algorithms or numerical methods. Furthermore, it is often not clear
how to sort the relevant contributions from multiple operators into
just two.

In this study, we focus on operator-splitting methods for
$\Nop$-additively split ordinary differential equations (ODEs)
\begin{equation}
	\label{eq:Nsplit_ODE}
	\dv{y}{t} = \mathcal{F}(t,y) = \sum\limits_{\ell=1}^{\Nop} \Fl{\ell}(t, y). 
\end{equation}
The most common operator-splitting methods have been derived $\Nop=2$
operators, possibly due to the significant increase in difficulty in
derivation and implementation with increasing $\Nop$. However, both
the first-order Lie--Trotter~\cite{Lie1888,Trotter1958} (or
Godunov~\cite{Godunov1959}) method and the second-order
Strang~\cite{Strang1968} (or Strang--Marchuk~\cite{Marchuk1971})
method can be applied directly to general $\Nop$-split problems. This
observation makes these methods natural choices for solving
$\Nop$-split problems. In this paper, we present a pair of
second-order operator-splitting methods that are also generalizable to
$\Nop$-split problems. The methods have complex-valued coefficients
with positive real parts and are conjugates of each other.

Operator-splitting methods are usually derived in two ways:
\begin{enumerate}
	\item Solve the order conditions directly. 
	\item Form compositions of existing methods.
\end{enumerate}

For $\Nop$-split problems, the number of equations in the order
conditions increases dramatically as $\Nop$ increases. Accordingly, it
is challenging to derive high-order $\Nop$-split methods by solving
the order conditions directly. When successful, any free parameters of
the method left over after the order conditions are satisfied can be
used to optimize certain method properties, such as leading error
measures or
stability~\cite{auzinger2014,auzinger2015,auzinger2016practical,blanes2013complex}.
Alternatively, if a (lower-order) method is known, it can be used as a
base method, from which higher-order methods can be derived through
composition~\cite{hairer2006,Hansen2009,castella2009}.  However, this
approach can be computationally expensive if the base method is
relatively expensive or if many compositions are required.

It is further known that operator-splitting methods with real
coefficients of order greater than two require backward-in-time
integration in each operator~\cite{goldman1996}. These backward
integrations can introduce instabilities for some applications, thus
reducing the appeal of such methods in practice. A possible remedy is
to use high-order complex-coefficient operator-splitting methods with
positive real parts. The performance of such methods compared with
real-valued operator-splitting methods has not been explored
extensively. Although complex-valued methods can have better stability
than real-valued methods with negative steps, the cost of complex
arithmetic can negate the stability advantages.

The goal of this paper is to propose a new pair of inexpensive
low-order base methods for $\Nop$-split problems and show their
effectiveness not only on their own but also as higher-order methods
obtained through a minimum number of compositions. The proposed
methods are second-order accurate, are conjugates of each other, and
have complex coefficients with positive real parts. We apply the
methods to real- and complex-valued differential equations and discuss
their performance.

The remainder of the paper is organized as follows. In
\cref{sec:complexLT_os}, we formally define $s$-stage, $\Nop$-split
operator-splitting methods and introduce necessary mathematical
background. In \cref{sec:complexLT_main}, we present the order
conditions of $\Nop$-split problems, derive a pair of second-order,
complex-valued methods for $\Nop$-split problems, and compose these
base methods to derive higher-order $\Nop$-split methods. In
\cref{sec:complexLT_num_exp}, we present two numerical examples using
the proposed methods and compare their performance with the Strang
method. In \cref{sec:complexLT_conclude}, we summarize our findings
and discuss potential future work.


\section{$\Nop$-split operator-splitting methods}
\label{sec:complexLT_os}

In this section, we give a formal definition of $\Nop$-split
operator-splitting methods and establish the notation used in the
following sections.

Let $\varphi_t^{[\ell]}$ be the exact flow of
$\displaystyle \dv{y^{[\ell]}}{t}= \Fl{\ell}(t,y^{[\ell]})$ for
$\ell=1,2,\dots, \Nop$. The majority of operator-splitting methods
focus on $\Nop = 2$. The first-order Lie--Trotter method to solve a
$2$-split ODE with a step size $\Dt$ can be written as the
composition of the flows,
\begin{equation*}
	\label{eq:LT}
	\Psi_{\Dt}^{\text{LT-2}} = \varphi_{\Dt}^{[2]} \circ \varphi_{\Dt}^{[1]}.
\end{equation*}

The second-order Strang method can be viewed as the composition of the
Lie--Trotter method with its adjoint over half steps:
\begin{equation*}
	\label{eq:Strang}
	\Psi_{\Dt}^{\text{S-N}} = \varphi_{\Dt/2}^{[1]} \circ \varphi_{\Dt/2}^{[2]} \circ \varphi_{\Dt/2}^{[2]} \circ \varphi_{\Dt/2}^{[1]} = \varphi_{\Dt/2}^{[1]} \circ \varphi_{\Dt}^{[2]}  \circ \varphi_{\Dt/2}^{[1]} .
\end{equation*} 
Both the Lie--Trotter method and the Strang methods can be generalized
to $\Nop$-split problems for arbitrary $\Nop$. Explicitly, the
$\Nop$-split Lie--Trotter method and Strang method are, respectively,
\begin{equation}
	\label{eq:LT_N} 
	\Psi_{\Dt}^{\text{LT-N}} = \varphi_{\Dt}^{[\Nop]}\circ \cdots \circ \varphi_{\Dt}^{[2]} \circ \varphi_{\Dt}^{[1]}.
\end{equation}

\begin{equation}
	\label{eq:Strang_N} 
	\Psi_{\Dt}^{\text{S-N}} =  \varphi_{\Dt/2}^{[1]}\circ \cdots \circ
	\varphi_{\Dt/2}^{[\Nop-1]} \circ \varphi_{\Dt}^{[\Nop]}\circ
	\varphi_{\Dt/2}^{[\Nop-1]} \circ \cdots \circ \varphi_{\Dt/2}^{[2]} \circ \varphi_{\Dt/2}^{[1]}.
\end{equation}

\begin{mydef}
	\label{def:os_n}
	Consider the $\Nop$-split ODE \cref{eq:Nsplit_ODE}. Let
	$\vphi{t}{\ell}$ be the exact flow of the subsystem
	$\displaystyle \dv{y^{[\ell]}}{t} = \Fl{\ell}(t,y^{[\ell]})$. We
	define an $s$-stage, $\Nop$-split operator-splitting method
	$\Psi_{\Dt}$ with step size $\Dt$ to be given by
	\begin{equation}
		\label{os_description}
		\Psi_{\Dt} = \vphi{\aaalpha{s}{\Nop}\Dt}{\Nop} \circ \cdots \circ \vphi{\aaalpha{s}{1}\Dt}{1} \circ \cdots \circ 
		\vphi{\aaalpha{1}{\Nop}\Dt}{\Nop} \circ \cdots \circ
		\vphi{\aaalpha{1}{1}\Dt}{1}, \qquad \ell=1,2,\dots,\Nop. 
	\end{equation}
	Hence, the operator-splitting method is determined by the
	coefficients
	$\{\aaalpha{k}{\ell}\}_{k=1,2,\dots,s}^{\ell = 1, 2,\dots, \Nop}$.
\end{mydef}

\begin{rmk}
	We note that the definition of a stage involves a complete cycle
	through the operators in order from $1$ to $\Nop$. However, cycling
	through the operators in an arbitrary order can be achieved (albeit
	at a nominal increase in the number of stages) by setting
	appropriate the coefficients $\aaalpha{k}{\ell}$ to zero.
\end{rmk}

Gr\"{o}bner's lemma can now be used to express the composition of the
exact flows in terms of exponential operators:
\begin{lem} \label{lem:Grobner} [Gr\"{o}bner 1960] Let
	$\varphi_s^{[1]}$ and $\varphi_t^{[2]}$ be the flows of the
	differential equations $\displaystyle \dv{\yl{1}}{t} = \Fl{1} (t, \yl{1})$ and
	$\displaystyle \dv{\yl{2}}{t} = \Fl{2} (t, \yl{2})$, respectively. For their
	composition, we then have
	\begin{equation*}
		(\varphi_{t}^{[2]} \circ \varphi_{s}^{[1]})(y_0) =  \exp(s
		\DOp{1})\exp(t \DOp{2}) \, y_0,
	\end{equation*}
	where
	$\displaystyle \DOp{\ell} = \sum\limits_{j} \Fl{\ell}_j \pdv{}{y_j}$
	is the Lie derivative.
\end{lem}

Applying Gr\"{o}bner's lemma repeatedly to composition
\cref{os_description}, we have
\begin{equation}
	\label{nsplit_exp}
	\Psi_{\Dt} (y_0)  =  \prod\limits_{k=1}^s \left(\prod\limits_{\ell = 1}^\Nop  \exp(\aaalpha{k}{\ell} \Dt \DOp{\ell} )\right) y_0,
\end{equation}
where with some abuse of notation, $\prod$ denotes composition.  In
general, the operators $\DOp{\ell}$ do not commute. Hence, the numerical
solution $\Psi_{\Dt}$ is different from the exact solution
$\varphi_{\Dt} = \exp(\Dt(\DOp{1} + \DOp{2} + \cdots + \DOp{\Nop}))$ of
$\displaystyle \dv{y}{t} = \mathcal{F}(t,y)$, and the order
conditions for a given method can be derived from the expansion of
the difference between the two.


\section{Main results}
\label{sec:complexLT_main}

Strictly speaking, Gr\"{o}bner's lemma is only stated for two
operators.  The relation \cref{nsplit_exp} suggests that a general
version of the Baker--Campbell--Hausdorff (BCH)
formula~\cite{hairer2006} is needed to derive order conditions for
$\Nop$-split operator-splitting methods. In \cref{subsec:BCH_N}, we
generalize the BCH formula to a product of $\Nop$ terms. In
\cref{subsec:order_cond_N}, we explicitly derive the order conditions
of an $s$-stage, $\Nop$-split operator-splitting method of up to order
two. In~\cref{subsec:deriveOSN}, we give the derivation of the new
two-stage, $\Nop$-split operator-splitting methods called
\emph{complex Lie--Trotter} (CLT) methods because they can be
interpreted as a composition of Lie--Trotter steps with complex
steps. These methods are second order and conjugates of each other. We
also propose and evaluate two third-order $\Nop$-split
operator-splitting methods based on composition as well as show how
even higher-order methods can be built in a similar fashion.

\subsection{The BCH formula for $\Nop$-split problems}
\label{subsec:BCH_N}

For non-commutative operators $X,Y$, it is well known that
$\exp(X)\exp(Y) \neq \exp(X+Y)$. The BCH formula proved that there
exists a $Z$ that can be expressed in terms of commutators of $X$ and
$Y$, such that
\begin{equation*}
	\exp(X) \exp(Y)= \exp(Z). 
\end{equation*}
In particular, the first few terms of $Z$ are
\begin{equation}
	\label{eq:BCH}
	Z = X+Y+\frac{1}{2}[X,Y] + \frac{1}{12} [X,[X,Y]] - \frac{1}{12}[Y,[X,Y]] + \cdots. 
\end{equation}
The original papers by Baker \cite{Baker1905}, Campbell
\cite{Campbell1897}, and Hausdorff \cite{Hausdorff1906} showed that
$Z$ can be written as nested commutators of $X$ and $Y$ but did not
give an explicit formula. In 1947, Dynkin \cite{dynkin1947} derived
an explicit formula for $Z$:
\begin{equation} \label{dynkin}
	Z = \sum\limits_{k=1}^\infty \sum\limits_{p_i,q_i}
	\frac{(-1)^{k-1}}{k} \frac{[X^{p_1}Y^{q_1} X^{p_2}Y^{q_2}\cdots
		X^{p_k}Y^{q_k}]}{\left(\sum\limits_{i=1}^k (p_i+q_i) \right)
		p_1! q_1! p_2! q_2! \cdots p_k! q_k!},
\end{equation} 
where the inner summation is taken over all non-negative integers
$p_i,q_i$ such that $p_i+q_i >0$ and
$[X^{p_1}Y^{q_1} X^{p_2}Y^{q_2}\cdots X^{p_k}Y^{q_k}] = \left[X^{p_1},
\left[Y^{q_1}, \left[X^{p_2}, \left[Y^{q_2}, \dots
\left[X^{p_{k-1}}, \left[Y^{q_{k-1}}, \left[X^{p_k},
Y^{q_k}\right]\right]\right] \dots
\right]\right]\right]\right]$ denotes the right-nested
commutator. However, formula \cref{dynkin} is difficult to use in
practice due to the complexity of the sum and the fact that the nested
commutators are not distinct according to the Jacobi identity
\begin{equation*}
	[X_1, [X_2, X_3]] + [X_2,[X_3,X_1]] + [X_3,[X_1,X_2]] = 0. 
\end{equation*}
Finding a succint expression for $Z$ is challenging. In 2009, Casas
and Murua~\cite{Casas2009} constructed an efficient algorithm to
generate an expression for $Z$.  We now generalize the BCH formula to
the case of $\Nop$ exponentials
$\exp(X_1), \exp(X_2),\dots, \exp(X_\Nop)$. In particular, we give the
first few terms in terms of nested commutators of
$X_1, X_2,\dots,X_{\Nop}$.




\begin{prop}{$\Nop$-term BCH formula} \label{prop:bch_n}{\ }
	
	For $N \geq 2$, let $\opp{1},\opp{2},\dots, \opp{N}$ be arbitrary
	(in general non-commuting) matrices. Then there exists a $Z_N(t)$
	such that the following relation holds:
	\begin{equation}
		\exp(t\opp{1}) \exp(t\opp{2}) \cdots \exp(t\opp{N}) = \exp(Z_N(t)).
	\end{equation}
	
	The first few terms of the Taylor coefficients of
	$Z_N(t) = t\bchcoeff{N}{1} + t^2 \bchcoeff{N}{2} + t^3
	\bchcoeff{N}{3} + \dots$ are given by
	\begin{equation*}\resizebox{\hsize}{!}{
			$\begin{aligned}
				\bchcoeff{N}{1} & = \sum\limits_{i=1}^N \opp{i},  \\
				\bchcoeff{N}{2} & = \frac{1}{2} \sum\limits_{i=1}^{N-1} \sum\limits_{j=i+1}^N 
				\LieTwo{\opp{i}}{\opp{j}}, \\ 
				\bchcoeff{N}{3} & = \frac{1}{12}\sum\limits_{i=1}^{N} \sum\limits_{j=1, j\neq i }^N 
				\LieThree{\opp{i}}{\opp{i}}{\opp{j}}+ \frac{1}{6}\sum\limits_{i=1}^{N-2} \sum\limits_{j=i+1}^{N-1} \sum\limits_{k=j+1}^N \left( \LieThree{\opp{i}}{\opp{j}}{\opp{k}}+ [[\opp{i},\opp{j}],\opp{k}]\right),
			\end{aligned}$}
	\end{equation*}
	where $\LieTwo{\opp{i}}{\opp{j}} = \opp{i}\opp{j} -
	\opp{j}\opp{i}$ is the Lie bracket.
\end{prop}

\begin{proof}
	First, the existence of $Z_N(t)$ is guaranteed by construction
	through repeated application of the original BCH
	formula~\cref{eq:BCH}. To find the first few terms of $Z_N(t)$, we
	proceed by induction.
	
	The base case with $N = 2$ coincides with the original BCH
	formula~\cref{eq:BCH}, which states
	\begin{align*}
		Z_2(t) & = \log(\exp(t\opp{1}) \exp(t\opp{2}) )  \\
		& = t(\opp{1} + \opp{2}) + t^2 \frac{1}{2} [\opp{1}, \opp{2}]  
		+ t^3 \frac{1}{12} (\LieThree{\opp{1}}{\opp{1}}{\opp{2}} + \LieThree{\opp{2}}{\opp{2}}{\opp{1}}) + \bigOh(t^4). 
	\end{align*}
	
	We now assume that
	$Z_{N-1}(t) = t\bchcoeff{N-1}{1} + t^2 \bchcoeff{N-1}{2} + t^3
	\bchcoeff{N-1}{3} + \bigOh(t^4)$, where
	\begin{equation*}\resizebox{\hsize}{!}{
			$\begin{aligned}
				\bchcoeff{N-1}{1} & = \sum\limits_{i=1}^{N-1} \opp{i}, \\
				\bchcoeff{N-1}{2} & = \frac{1}{2} \sum\limits_{i=1}^{N-2} \sum\limits_{j=i+1}^{N-1} 
				[\opp{i},\opp{j}], \\ 
				\bchcoeff{N-1}{3} & =  \frac{1}{12}\sum\limits_{i=1}^{N-1} \sum\limits_{j=1, j\neq i }^{N-1} 
				[\opp{i},[\opp{i},\opp{j}]] + \frac{1}{6}\sum\limits_{i=1}^{N-3} \sum\limits_{j=i+1}^{N-2} \sum\limits_{k=j+1}^{N-1} \left( [\opp{i},[\opp{j},\opp{k}]] + [[\opp{i},\opp{j}],\opp{k}]\right).
			\end{aligned}$}
	\end{equation*}

	We now compute $Z_{N}(t)$:
	
	\begin{equation}  \label{eq:ZN}
		\begin{aligned}
			Z_{N}(t) & = \log\left[ \exp(t\opp{1}) \exp(t\opp{2}) \cdots \exp(t\opp{N}) \right] & \\
			& = \log\left[\exp(t \frac{Z_{N-1}(t)}{t}) \exp(t\opp{N}) \right]  \\
			& = t \left(\frac{Z_{N-1}(t)}{t} + \opp{N}\right) +\frac{1}{2} t^2 \left[\frac{Z_{N-1}(t)}{t}, \opp{N}\right] \\
			& \quad + \frac{1}{12}t^3 \left(\left[\frac{Z_{N-1}(t)}{t} ,
			\left[\frac{Z_{N-1}(t)}{t},\opp{N} \right]\right] + \left[\opp{N},
			\left[\opp{N},\frac{Z_{N-1}(t)}{t}\right]\right] \right) + \bigOh(t^4).
		\end{aligned}
	\end{equation}
	
	Inserting the expression for $Z_{N-1}(t)$ into \cref{eq:ZN} and
	collecting the coefficients of $t, t^2,$ and $t^3$ yields
	
	\begin{align*}
		\bchcoeff{N}{1} & = \bchcoeff{N-1}{1} + \opp{N} = \sum\limits_{i=1}^{N-1} \opp{i} + \opp{N} = \sum\limits_{i=1}^{N} \opp{i}. 
	\end{align*}
	
	\begin{align*}
		\bchcoeff{N}{2} & = \bchcoeff{N-1}{2} + \frac{1}{2} [\bchcoeff{N-1}{1}, \opp{N}] \\
		& = \frac{1}{2} \sum\limits_{i=1}^{N-2} \sum\limits_{j=i+1}^{N-1} 
		[\opp{i},\opp{j}]  + \frac{1}{2}\left[\sum\limits_{i=1}^{N-1} \opp{i}, \opp{N}\right] \\
		& =  \frac{1}{2} \sum\limits_{i=1}^{N-1} \sum\limits_{j=i+1}^{N} 
		[\opp{i},\opp{j}].  
	\end{align*}
	
	\begin{equation*}\resizebox{\hsize}{!}{
			$\begin{aligned}
				\bchcoeff{N}{3}	& = \bchcoeff{N-1}{3} + \frac{1}{2}\LieTwo{\bchcoeff{N-1}{2}}{\opp{N}} + \frac{1}{12} \left( \LieThree{\bchcoeff{N-1}{1}}{\bchcoeff{N-1}{1}}{\opp{N}} + \LieThree{\opp{N}}{\opp{N}}{\bchcoeff{N-1}{1}} \right)	\\
				& =  \frac{1}{12}\sum\limits_{i=1}^{N-1} \sum\limits_{j=1, j\neq i }^{N-1} 
				[\opp{i},[\opp{i},\opp{j}]] + \frac{1}{6}\sum\limits_{i=1}^{N-3} \sum\limits_{j=i+1}^{N-2} \sum\limits_{k=j+1}^{N-1} \left( [\opp{i},[\opp{j},\opp{k}]] + [[\opp{i},\opp{j}],\opp{k}]\right) \\
				& \quad +  \frac{1}{2} \LieTwo{ \frac{1}{2} \sum\limits_{i=1}^{N-2} \sum\limits_{j=i+1}^{N-1} 
					[\opp{i},\opp{j}]}{\opp{N}} + \frac{1}{12} 
				\left( \LieThree{ \sum\limits_{i=1}^{N-1} \opp{i}}{ \sum\limits_{j=1}^{N-1} \opp{j} }{\opp{N}} +  \LieThree{\opp{N}}{\opp{N}}{  \sum\limits_{i=1}^{N-1} \opp{i} } 
				\right) \\ 
				& = \frac{1}{12} \sum\limits_{i=1}^{N} \sum\limits_{j=1, j\neq i }^{N} 
				[\opp{i},[\opp{i},\opp{j}]] + \frac{1}{6}\sum\limits_{i=1}^{N-3} \sum\limits_{j=i+1}^{N-2} \sum\limits_{k=j+1}^{N-1} \left( [\opp{i},[\opp{j},\opp{k}]] + [[\opp{i},\opp{j}],\opp{k}]\right) \\
				& \quad + \frac{1}{6} \sum\limits_{i=1}^{N-2} \sum\limits_{j=i+1}^{N-1}  \LieTwo{ 
					[\opp{i},\opp{j}]}{\opp{N}} +  \frac{1}{12} \sum\limits_{i=1}^{N-2} \sum\limits_{j=i+1}^{N-1} \LieTwo{ 
					[\opp{i},\opp{j}]}{\opp{N}}  +\frac{1}{12}  \sum\limits_{i=1}^{N-1}  \sum\limits_{j=1, j
					\neq i}^{N-1}  \LieThree{  \opp{i}}{ \opp{j} }{\opp{N}} \\ 
				& = \frac{1}{12} \sum\limits_{i=1}^{N} \sum\limits_{j=1, j\neq i }^{N} 
				[\opp{i},[\opp{i},\opp{j}]] + \frac{1}{6}\sum\limits_{i=1}^{N-3} \sum\limits_{j=i+1}^{N-2} \sum\limits_{k=j+1}^{N-1} \left( [\opp{i},[\opp{j},\opp{k}]] + [[\opp{i},\opp{j}],\opp{k}]\right) \\
				& \quad + \frac{1}{6} \sum\limits_{i=1}^{N-2} \sum\limits_{j=i+1}^{N-1}  \LieTwo{ 
					[\opp{i},\opp{j}]}{\opp{N}}   +\frac{1}{6} \sum\limits_{i=1}^{N-2}  \sum\limits_{j=i+1}^{N-1}  \LieThree{  \opp{i}}{ \opp{j} }{\opp{N}} - \frac{1}{12}  \sum\limits_{i=1}^{N-2}  \sum\limits_{j=i+1}^{N-1}  \LieThree{  \opp{i}}{ \opp{j} }{\opp{N}} \\
				& \quad + \frac{1}{12} \sum\limits_{i=1}^{N-2} \sum\limits_{j=i+1}^{N-1} \LieTwo{ 
					[\opp{i},\opp{j}]}{\opp{N}}   
				+ \frac{1}{12} \sum\limits_{i=1}^{N-2}  \sum\limits_{j=1}^{i-1}  \LieThree{  \opp{i}}{ \opp{j} }{\opp{N}} + \frac{1}{12} \sum\limits_{i=N-1}  \sum\limits_{j=1}^{N-2}  \LieThree{  \opp{i}}{ \opp{j} }{\opp{N}}\\ 
				& = \frac{1}{12} \sum\limits_{i=1}^{N} \sum\limits_{j=1, j\neq i }^{N} 
				[\opp{i},[\opp{i},\opp{j}]] + \frac{1}{6}\sum\limits_{i=1}^{N-2} \sum\limits_{j=i+1}^{N-1} \sum\limits_{k=j+1}^{N} \left( [\opp{i},[\opp{j},\opp{k}]] + [[\opp{i},\opp{j}],\opp{k}]\right) \\
				& \quad +\frac{1}{12} \sum\limits_{i=1}^{N-2} \sum\limits_{j=i+1}^{N-1} \left( - \LieThree{  \opp{i}}{ \opp{j} }{\opp{N}} + \LieTwo{ 
					[\opp{i},\opp{j}]}{\opp{N}}   +   \LieThree{\opp{j}}{ \opp{i} }{\opp{N}}\right) \\ 
				& = \frac{1}{12} \sum\limits_{i=1}^{N} \sum\limits_{j=1, j\neq i }^{N} 
				[\opp{i},[\opp{i},\opp{j}]] + \frac{1}{6}\sum\limits_{i=1}^{N-2} \sum\limits_{j=i+1}^{N-1} \sum\limits_{k=j+1}^{N} \left( [\opp{i},[\opp{j},\opp{k}]] + [[\opp{i},\opp{j}],\opp{k}]\right).
			\end{aligned}$}
	\end{equation*}
	
	Hence, we obtain the desired result.
\end{proof}

\subsection{Order conditions via the generalized BCH formula}
\label{subsec:order_cond_N}

In this section, we derive the order conditions of $s$-stage,
$\Nop$-split operator-splitting methods up to order $2$. We apply the
generalized BCH formula in \cref{prop:bch_n} to \cref{nsplit_exp} to
get one exponential series in powers of $\Dt$. Then we compare this
series with the exact solution $\exp(\Dt(\DOp{1}+\DOp{2} +\cdots+\DOp{\Nop}))$.

Stage $k$ of the $\Nop$-split operator-splitting method \cref{nsplit_exp}
is a composition of the form
\begin{equation}
	\vphi{\aaalpha{k}{\Nop}\Dt}{\Nop} \circ
	\cdots \circ \vphi{\aaalpha{k}{1}\Dt}{1} = \exp(\aaalpha{k}{1}\Dt \DOp{1})  \exp(\aaalpha{k}{2}\Dt \DOp{2}) \cdots \exp(\aaalpha{k}{\Nop}\Dt \DOp{\Nop}).   
\end{equation}
Applying \cref{prop:bch_n}, we obtain
\begin{equation}\label{stagek_exp}\resizebox{\hsize}{!}{
		$\begin{aligned}
			\vphi{\aaalpha{k}{\Nop}\Dt}{\Nop} \circ \cdots \circ \vphi{\aaalpha{k}{1}\Dt}{1}  = \exp & \left(  
			\Dt \sum\limits_{\ell=1}^{\Nop} \aaalpha{k}{\ell} \DOp{\ell}  + \frac{1}{2} \Dt^2\sum\limits_{\ell_1=1}^{\Nop-1} \sum\limits_{\ell_2=\ell_1+1}^{\Nop} \aaalpha{k}{\ell_1} \aaalpha{k}{\ell_2} \LieTwo{\DOp{\ell_1}}{\DOp{\ell_2}} \right. \\ 
			& \quad + \frac{1}{12} \Dt^3 \sum\limits_{\ell_1=1}^{\Nop} \sum\limits_{\ell_2=1, \ell_2\neq \ell_1 }^{\Nop}  (\aaalpha{k}{\ell_1})^2 \aaalpha{k}{\ell_2} \LieThree{\DOp{\ell_1}}{\DOp{\ell_1}}{\DOp{\ell_2}} \\ 
			& \quad + \frac{1}{6} \Dt^3   \sum\limits_{\ell_1=1}^{\Nop-2} \sum\limits_{\ell_2=\ell_1+1}^{\Nop-1} \sum\limits_{\ell_3=\ell_2+1}^{\Nop}   \aaalpha{k}{\ell_1} \aaalpha{k}{\ell_2} \aaalpha{k}{\ell_3} \LieThree{\DOp{\ell_1}}{\DOp{\ell_2}}{\DOp{\ell_3}} \\ 
			& \quad \left. + \frac{1}{6} \Dt^3   \sum\limits_{\ell_1=1}^{\Nop-2} \sum\limits_{\ell_2=\ell_1+1}^{\Nop-1} \sum\limits_{\ell_3=\ell_2+1}^{\Nop}   \aaalpha{k}{\ell_1} \aaalpha{k}{\ell_2} \aaalpha{k}{\ell_3} \LieTwo{[\DOp{\ell_1},\DOp{\ell_2}]}{\DOp{\ell_3}} + \bigOh(\Dt^4) \right). \\ 
		\end{aligned}$ } 
\end{equation}
Let $\Psi^{(k)}_{\Dt}$ be the numerical solution after stage $k$. $\Psi^{(k)}_{\Dt}$ is defined recursively by 
\begin{equation} \label{eq:psi_recursive} \Psi^{(0)}_{\Dt} = \text{Id}, \:
	\Psi^{(k)}_{\Dt} = \vphi{\aaalpha{k}{\Nop}\Dt}{\Nop} \circ \cdots \circ
	\vphi{\aaalpha{k}{1}\Dt}{1} \Psi^{(k-1)}_{\Dt},
\end{equation}
such that $\Psi^{(s)}_{\Dt} = \Psi_{\Dt}$ as defined in
\cref{os_description}. We aim to write $\Psi^{(k)}_{\Dt}$ as an exponential
operator.
\begin{lem} 
	\label{lem:recursive_c}
	The method $\Psi^{(k)}_{\Dt}$ defined in \cref{eq:psi_recursive} can be written as 
	\begin{equation}\label{psik_exp}\resizebox{\hsize}{!}{
			$\begin{aligned}
				\Psi^{(k)}_{\Dt} = \exp & \left(\Dt \sum\limits_{\ell=1}^{\Nop} c_{\ell,k}^1 \DOp{\ell} +  \Dt^2\sum\limits_{\ell_1=1}^{\Nop-1} \sum\limits_{\ell_2=\ell_1+1}^{\Nop} c_{\ell_1,\ell_2,k}^2 \LieTwo{\DOp{\ell_1}}{\DOp{\ell_2}} + \bigOh(\Dt^3) \right) \text{Id}
			\end{aligned}$ }
	\end{equation}  
	where all coefficients are zero for $k=0$ and 
	\begin{align}
		c_{\ell, k}^1 & = c_{\ell, k-1}^1 + \aaalpha{k}{\ell}, \\ 
		c_{\ell_1,\ell_2,k}^2 & = c_{\ell_1,\ell_2,k-1}^2  + \frac{1}{2} \aaalpha{k}{\ell_1} \aaalpha{k}{\ell_2}  + \frac{1}{2} \aaalpha{k}{\ell_1} c_{\ell_2,k-1}^1 - \frac{1}{2} c_{\ell_1,k-1}^1\aaalpha{k}{\ell_2}, 
	\end{align}
	for $k=1,2,\dots$.
	
\end{lem}

\begin{proof}
	$\Psi^{(k)}_{\Dt} = \exp(A) \exp(B)$ where $A$ is the argument of the
	exponential form in \cref{stagek_exp} and $B$ is the argument of the
	exponential form in \cref{psik_exp} with $k$ substituted by $k-1$.
	
	The coefficients of $\DOp{\ell}$ in $\Psi^{(k)}_{\Dt}$ can only have contributions
	from the coefficients of $\DOp{\ell}$ in $A$ and $B$. Hence,
	\begin{equation*}
		c_{\ell, k}^1  = c_{\ell, k-1}^1 + \aaalpha{k}{\ell} .
	\end{equation*}

	The coefficients of $\LieTwo{\DOp{\ell_1}}{\DOp{\ell_2}}$ in
	$\Psi^{(k)}_{\Dt}$ come from the coefficients of
	$\LieTwo{\DOp{\ell_1}}{\DOp{\ell_2}}$ in $A$ and $B$ and the term
	$\displaystyle \frac{1}{2}[A,B]$ in the product
	$\exp(A)\exp(B)$. Therefore,
	\begin{equation*}
		c_{\ell_1,\ell_2,k}^2  = c_{\ell_1,\ell_2,k-1}^2  + \frac{1}{2} \aaalpha{k}{\ell_1} \aaalpha{k}{\ell_2} +  \frac{1}{2} \aaalpha{k}{\ell_1} c_{\ell_2,k-1}^1 - \frac{1}{2} c_{\ell_1,k-1}^1\aaalpha{k}{\ell_2}.
	\end{equation*}
	
	%
	%
	%
\end{proof}

\begin{rmk} \label{rmk:order_cond} Similar to Theorem 5.6 in
	\cite{hairer2006}, the operator-splitting method
	\cref{os_description} is of order $2$ if
	\begin{align*}
		c_{\ell, s}^1  &= 1, && \text{ for } \ell = 1,2,\dots, \Nop, \\ 
		c_{\ell_1,\ell_2,s}^2 &= 0, && \text{ for } \ell_1, \ell_2 = 1,2,\dots, \Nop \text{ and } \ell_1 \neq \ell_2, 
	\end{align*}
	where the coefficients $c_{\ell, s}$ and $c_{\ell_1,\ell_2,s}^2$
	are those defined in \cref{lem:recursive_c}. It can be shown that
	\begin{align*}
		c_{\ell, s}^1 & = \sum\limits_{k=1}^s \aaalpha{k}{\ell},  \\ 
		c_{\ell_1,\ell_2, s}^2 & = \frac{1}{2} \sum\limits_{k=1}^s \left[\aaalpha{k}{\ell_1}\left( \sum\limits_{j=1}^k \aaalpha{j}{\ell_2} \right) -\aaalpha{k}{\ell_2} \left( \sum\limits_{j=1}^{k-1} \aaalpha{j}{\ell_1} \right) \right] . 
	\end{align*}	
	The order conditions are thus 
	\begin{subequations}
		\begin{align*} 
			&p = 1: && \displaystyle \sum\limits_{k=1}^s
			\osalpha{k}{\ell}= 1,&& \text{for } \ell = 1,
			2, \dots, \Nop. \\ 
			&p = 2: && \displaystyle \sum\limits_{k=2}^{s}
			\osalpha{k}{\ell_1}
			\left(\sum\limits_{j=1}^{k-1}
			\osalpha{j}{\ell_2} \right)= \frac{1}{2},&&
			\text{for } \ell_1 = 1, 2,\dots, \Nop-1, \ \ell_2 = \ell_1+1, \ell_1+2, \dots, \Nop. 
		\end{align*}
	\end{subequations}
	There are $\Nop$ conditions at order 1 and
	$\displaystyle \frac{\Nop(\Nop-1)}{2}$ conditions at order 2 for a
	total of $\displaystyle \frac{\Nop(\Nop+1)}{2}$ conditions.
\end{rmk}

\begin{rmk}
	Order conditions for $p\geq 3$ can be obtained by matching the
	coefficients of the nested commutators to those of the exact
	solution. The complexity of these expressions increases dramatically as
	$p$ increases with arbitrary $\Nop$.  These order conditions are
	unwieldy to display explicitly, but they can be generated using
	computer algebra. Auzinger et
	al.~\cite{auzinger2016practical} have produced a MAPLE program that
	constructs order conditions for $s$-stage, $3$-split problems up to
	order $6$. The program requires manual input of the Lyndon words
	that represent the nested commutators of three operators.
\end{rmk}

\subsection{Derivation of $\Nop$-split operator-splitting methods}
\label{subsec:deriveOSN}

We now discuss the derivation of $\Nop$-split operator-splitting
methods for arbitrary positive integer $\Nop$.  As discussed in
\cite{spiteri2023_3split}, fewer sub-integrations generally lead to
more efficient methods. Therefore, we are interested in methods with a
low number of stages. The only one-stage, $\Nop$-split method of order
one is the Lie--Trotter method. It is not possible to achieve second
order with a one-stage method. There are, however, second-order
methods with two stages.

\begin{thm}\label{thm:OSNNS2P2}
	The only two-stage, second-order, $\Nop$-split operator-splitting
	methods for $\Nop > 2$ have coefficients given
	in~\cref{tab:OSNNS2P2}.
	\begin{table}[htbp]
		\centering
		\renewcommand{\arraystretch}{2.5} 
		\begin{tabular}{|c|c|c|}
			\hline 
			$k$ & $\displaystyle \{\alpha_k^{[\ell]}\}_{\ell=1, 2,\dots, \Nop}$   \\
			\hline 
			1 & $\displaystyle \frac{1}{2} \pm \frac{i}{2}$ \\
			\hline 
			2 &  $\displaystyle \frac{1}{2} \mp \frac{i}{2}$  \\
			\hline 
		\end{tabular}
		\caption{Coefficients $\osalpha{k}{\ell}$ for the two-stage,
			second-order, $\Nop$-split methods. }
		\label{tab:OSNNS2P2}
	\end{table}
\end{thm}

\begin{proof}
	Using \cref{rmk:order_cond}, the order conditions for two-stage,
	second-order, $\Nop$-split operator-splitting methods are
	\begin{subequations}
		\begin{align}
			& \osalpha{1}{\ell} + \osalpha{2}{\ell} = 1,  \text{ for } \ell = 1, 2,\dots, \Nop, \label{order_1_eq}\\ 
			&  \osalpha{2}{\ell_1}\osalpha{1}{\ell_2} = \frac{1}{2},
			\text{ for } \ell_1 = 1, 2,\dots, \Nop-1,\ \ell_2 = \ell_1+1, \ell_1 +2,\dots, \Nop. 
			\label{order_2_eq}
		\end{align}
	\end{subequations}
	
	We note that $\osalpha{2}{\ell_1} \neq 0$ for all
	$\ell_1 = 1,2,\dots, \Nop-1$; otherwise, the order conditions
	\cref{order_2_eq} cannot be satisfied. Therefore, using the
	conditions
	\begin{equation*}
		\displaystyle \osalpha{2}{1}\osalpha{1}{\ell_2} = \frac{1}{2}, \text{ for } \ell_2 = 2,3,\dots, \Nop,
	\end{equation*}
	we conclude that 
	\begin{equation*}
		\osalpha{1}{2} = \osalpha{1}{3} = \cdots = \osalpha{1}{\Nop} =: \gamma.
	\end{equation*}
	Similarly, using the conditions 
	\begin{equation*}
		\displaystyle \osalpha{2}{\ell_1}\osalpha{1}{\Nop} = \frac{1}{2}, \text{ for } \ell_1 = 1,2,\dots, \Nop-1,
	\end{equation*}
	we conclude that 
	\begin{equation*}
		\osalpha{2}{1} = \osalpha{2}{2} = \cdots = \osalpha{2}{\Nop-1} =: \eta.
	\end{equation*}
	Moreover, with $\Nop>2$, the first-order conditions \cref{order_1_eq} imply 
	\begin{equation} \label{eq:order_1_derive}
		\left\{\begin{aligned}
			\osalpha{1}{1} + \eta &= 1,\\ 
			\gamma + \eta &= 1, \\
			\gamma + \osalpha{2}{\Nop} &= 1.
		\end{aligned}   \right.
	\end{equation}  
	Hence, $\osalpha{1}{1} = \gamma$ and $\osalpha{2}{\Nop} = \eta$. 
	Using \cref{order_1_eq} and \cref{order_2_eq},
	\begin{equation} \label{simplified_ord_cond}
		\left\{\begin{aligned}
			\gamma + \eta &= 1, \\
			\eta  \gamma &= \frac{1}{2} .
		\end{aligned}   \right.
	\end{equation}  
	Solving \cref{simplified_ord_cond}, we get the following two sets
	of solutions:
	\begin{equation}
		\left\{ \begin{array}{l}
			\displaystyle \gamma = \frac{1}{2} + \frac{i}{2}, \\[1.5ex]
			\displaystyle \eta = \frac{1}{2} - \frac{i}{2}, 
		\end{array}   \right.
		\text{ and } 
		\left\{ \begin{array}{l}
			\displaystyle \gamma = \frac{1}{2} - \frac{i}{2}, \\[1.5ex]
			\displaystyle \eta = \frac{1}{2} + \frac{i}{2}. 
		\end{array}   \right.
	\end{equation}
	
\end{proof}

%
%

\begin{rmk}
	We note that the second equation in \cref{eq:order_1_derive} does
	not hold for $\Nop=2$. In fact, solving the order conditions for
	two-stage, second-order, $2$-split methods leads to a one-parameter
	family of solutions as shown in \cref{tab:OSN2S2P2}
	\cite{spiteri_wei_FSRK,wei_spiteri_pppds}.
	\begin{table}[htbp]
		\centering
		\renewcommand{\arraystretch}{2.5} 
		\begin{tabular}{|c|c|c|}
			\hline 
			$k$ & $\displaystyle \alpha_k^{[1]}$ & $\displaystyle \alpha_k^{[2]}$ \\
			\hline 
			1 & $\displaystyle \frac{2\OS22b-1}{2\OS22b-2}$ &  $\displaystyle 1-\OS22b$\\
			\hline 
			2 & $\displaystyle \frac{1}{2\OS22b-2}$ & $\OS22b$ \\
			\hline 
		\end{tabular}
		\caption{Coefficients $\osalpha{k}{\ell}$ for the
			one-parameter family of two-stage, second-order, $2$-split
			methods. }
		\label{tab:OSN2S2P2}
	\end{table}
\end{rmk}

\begin{rmk}
	The methods in~\cref{tab:OSNNS2P2} can be seen as a composition of
	the Lie--Trotter method over complex steps
	$\displaystyle \frac{1}{2} \pm \frac{i}{2}$ and
	$\displaystyle \frac{1}{2} \mp \frac{i}{2}$. Accordingly, we refer
	to the method with the upper signs as the Complex Lie--Trotter-2
	(\CLTTwo) method. We refer to the conjugate method as
	\CLTTwoConj.
\end{rmk}

\begin{rmk}
	The \CLTTwo\ method for $\Nop=3$ appeared
	in~\cite{Auzinger2015_3split}, but there was no indication that it
	generalized to arbitrary $\Nop$.
\end{rmk}

As mentioned, real-valued $N$-split operator-splitting methods of
order $p \geq 3$ require backward-in-time steps in every operator. A
common technique to avoid such steps is to use complex-valued
operator-splitting methods with positive real parts. Several works,
e.g.,~\cite{castella2009,blanes2013complex,chambers2003}, propose
composition using complex coefficients to derive complex-valued
operator-splitting methods. The main theorem used to construct
high-order methods through composition is given in \cite{hairer2006}:
\begin{thm} \label{thm:composition}
	Let $\basemethod{\Dt}{p}$ be a one-step method of order $p$. If 
	\begin{equation} \label{comp_cond}
		\begin{aligned}
			\sigma_{p,1} + \cdots + \sigma_{p,m} & = 1, \\ 
			\sigma_{p,1}^{p+1} + \cdots + \sigma_{p,m}^{p+1} & = 0,
		\end{aligned}
	\end{equation}
	then the $m$-term composition method 
	\begin{equation} \label{m_term_comp}
		\newmethod{q} = \basemethod{\sigma_{p,m} \Dt}{p} \circ \cdots \circ \basemethod{\sigma_{p,1} \Dt}{p}  
	\end{equation}
	is of order $q \geq p+1$. 
\end{thm}
To minimize the computational cost of an $m$-term composition
\cref{m_term_comp}, we only consider two-term compositions in this
paper. In particular, we consider the two-term compositions
proposed in \cite{Hansen2009}. As stated in theorem 2.2 in
\cite{Hansen2009}, given a base method $\Psi_{\Dt}$ of order two,
one can compose higher-order methods $\Psi_{\Dt}(p)$ with positive
real parts of order $p$, $p=3,4,5,6$, using the following chain of
two-term compositions
\begin{equation} \label{eq:composition} \Psi_{\Dt}(p) =
	\Psi_{\sigma_{p,2} \Dt} (p-1) \Psi_{\sigma_{p,1}\Dt} (p-1),
\end{equation}
where $\Psi_{\Dt} (2) = \Psi_{\Dt}$ and the coefficients
$\sigma_{p,1}$ and $\sigma_{p,2}$ are
\begin{equation} \label{eq:composition_coeff} \sigma_{p,1} =
	\frac{1}{2} + i \frac{\sin(\pi/p)}{2+2\cos(\pi/p)} \text{ and }
	\sigma_{p,2} = \overline{\sigma}_{p,1}.
\end{equation} 
The coefficients satisfy \cref{comp_cond}. Because the arguments
(phases) of the coefficients $\sigma_{p,i}$ for $i=1,2$ are added
together in each composition, this method does not yield
operator-splitting methods with positive real part of arbitrarily
high order. Hence, it is important to choose $\sigma_{p,i}$ with
small argument. As stated in \cite{Hansen2009}, the coefficients
\cref{eq:composition_coeff} satisfy \cref{comp_cond} and have the
smallest positive argument.  The strategy of achieving higher
order through the choice of coefficients from the two-term
composition \cref{eq:composition} results in $\Psi_{\Dt}(p)$ with
coefficients having negative real parts for $p\geq 7$.

For the purposes of this study, we only consider third-order
methods derived from a two-term composition~\cref{eq:composition}
of second-order methods. The third-order method derived from the
two-term composition of \CLTTwo\ is denoted by \CLTThree\ and has
coefficients given in \cref{tab:OSNNS4P3}. The third-order method
derived from the two-term composition~\cref{eq:composition} of
Strang is denoted by \CStrangThree\ and has coefficients given in
\cref{tab:complex_strang_p3}.

\begin{table}[htbp]
	\centering
	\renewcommand{\arraystretch}{2.5} 
	\begin{tabular}{|c|c|}
		\hline 
		$k$ & $\displaystyle \{\alpha_k^{[\ell]}\}_{\ell = 1, 2,\dots, \Nop}$ \\
		\hline 
		1 & $\displaystyle \frac{1}{4} - \frac{1}{4\sqrt{3}} + \left(\frac{1}{4} + \frac{1}{4\sqrt{3}}\right)i$  \\
		\hline 
		2 &  $\displaystyle  \frac{1}{4} + \frac{1}{4\sqrt{3}} + \left(-\frac{1}{4} + \frac{1}{4\sqrt{3}}\right)i$  \\
		\hline 
		3 & $\displaystyle \frac{1}{4} + \frac{1}{4\sqrt{3}} + \left(\frac{1}{4} - \frac{1}{4\sqrt{3}}\right)i$ \\
		\hline 
		4 &  $\displaystyle  \frac{1}{4} - \frac{1}{4\sqrt{3}} + \left(-\frac{1}{4} - \frac{1}{4\sqrt{3}}\right)i$ \\
		\hline 
	\end{tabular}
	\caption{Coefficients $\osalpha{k}{\ell}$ for the four-stage,
		third-order, $\Nop$-split method derived from \CLTTwo. }
	\label{tab:OSNNS4P3}
\end{table}

\begin{table}[htbp]
	\centering
	\renewcommand{\arraystretch}{2.5} 
	\begin{tabular}{|c|c|c|c|c|c|}
		\hline 
		$k$ & $\displaystyle \alpha_k^{[1]}$ & $\displaystyle \alpha_k^{[2]}$ &  $\cdots$ &  $\displaystyle \alpha_k^{[\Nop-1]}$ &   $\displaystyle \alpha_k^{[\Nop]}$ \\
		\hline 
		1 & $\displaystyle\frac{\sigma_{3,1}}{2}$ & $\displaystyle\frac{\sigma_{3,1}}{2}$  &  $\cdots $ & $\displaystyle\frac{\sigma_{3,1}}{2}$ & $\sigma_{3,1}$ \\
		\hline 
		2 &  $0$ & $0$ & $\cdots$ &  $\displaystyle\frac{\sigma_{3,1}}{2}$ & $0$  \\
		\hline 
		$\vdots$ &   &  & $\iddots$ &  &   \\
		\hline 
		$\Nop$ & $\displaystyle \frac{\sigma_{3,1}+\sigma_{3,2}}{2}$ & $\displaystyle \frac{\sigma_{3,2}}{2}$  &  $\cdots $ & $\displaystyle \frac{\sigma_{3,2}}{2}$ & $\sigma_{3,2}$ \\
		\hline 
		$\vdots$ &   &  & $\iddots$ &  &   \\
		\hline 
		$2\Nop -1$ &  $\displaystyle \frac{\sigma_{3,2}}{2}$ & $0$ & $\cdots$  &   $\cdots$  &  $0$ \\
		\hline 
	\end{tabular}
	\caption{Coefficients $\osalpha{k}{\ell}$ for the
		$(2\Nop-1)$-stage, third-order, $\Nop$-split method derived
		from Strang with constants
		$\sigma_{3,1} = \frac{1}{2} + \frac{\sqrt{3}}{6}i $ and
		$\sigma_{3,2} = \frac{1}{2} - \frac{\sqrt{3}}{6}i $.}
	\label{tab:complex_strang_p3}
\end{table}


\section{Numerical experiments}
\label{sec:complexLT_num_exp}

In this section, we give two numerical examples using the second-order
$\Nop$-split \CLTTwo\ methods (given in \cref{tab:OSNNS2P2}) and the
two third-order $\Nop$-split methods (given in \cref{tab:OSNNS4P3} and
\cref{tab:complex_strang_p3}) obtained from a two-term composition of
the second-order \CLTTwo\ and Strang methods, respectively. We verify
their orders of convergence and compare their performance on real- and
complex-valued differential equations against real-valued
operator-splitting methods.

%

\subsection{2D ADR problem}

We first consider a two-dimensional advection-diffusion-reaction
problem,

\begin{equation} \label{eq:ADR} 
	\pdv{u}{t} = - \alpha\, \nabla \cdot u 
	+ \epsilon\,	\nabla^2 u + \gamma u \left(u-\frac{1}{2}\right)(1-u),
\end{equation}
with homogeneous Neumann boundary conditions, initial conditions
\begin{equation*}
	u(x, y, 0) = 256\,(xy(1-x)(1-y))^2 + 0.3,
\end{equation*}
and solved on $t \in [0, 0.1]$, $(x, y) \in [0, 1]^2$ with parameters
$\alpha = -10$, $\epsilon=1/100$, $\gamma = 100$.  For the spatial
discretization, we use central finite differences on a uniform grid
with grid size $\displaystyle \Delta x = \Delta y = 1/40$. A reference
solution of the resulting method-of-lines ODEs is obtained using
MATLAB's {\tt ode45} function with {\tt atol = 3e-16 } and {\tt rtol =
	3e-16}. The \cref{eq:ADR} is split into $4$ operators:
\begin{equation}
	\label{eq:ADR_split}
	\begin{aligned}
		& \pdv{u^{[1]}}{t} = - \alpha\, \nabla \cdot u^{[1]}, \\
		& \pdv{u^{[2]}}{t} = \epsilon \pdv[2]{u^{[2]}}{x}, \\
		& \pdv{u^{[3]}}{t} = \epsilon \pdv[2]{u^{[3]}}{y}, \\
		& \pdv{u^{[4]}}{t} = \gamma u^{[4]} \left(u^{[4]}-\frac{1}{2}\right)(1-u^{[4]}),
	\end{aligned}
\end{equation}
and solved using operator-splitting methods: Strang, \CLTTwo,
\CLTThree, and \CStrangThree, while the sub-systems are solved using
the classical fourth-order explicit Runge--Kutta method (RK4). The
error of the numerical solution is measured using the $\ell_2$ norm of
the error between the reference solution and the numerical solution at
$t_\text{f} = 0.1$. As shown in \cref{fig:adr_convergence}, all four
methods exhibit the correct order of convergence. As shown in the
work-precision diagram \cref{fig:adr_wp}, both third-order methods
\CLTThree\ and \CStrangThree\ are more efficient than their
second-order counterparts for tighter tolerances.

\begin{figure}[htbp]
	\centering
	\includegraphics[width = \textwidth]{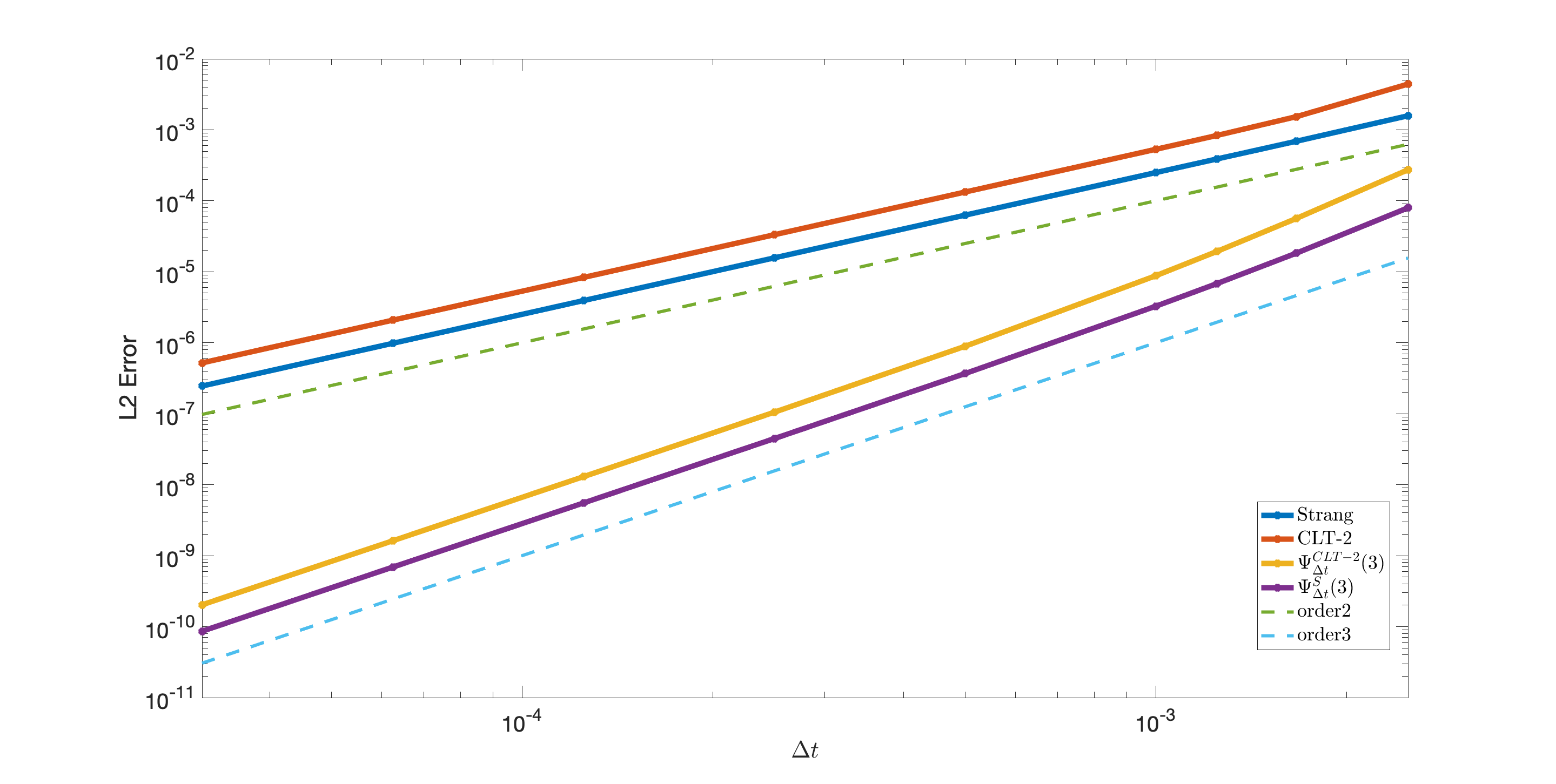}
	\caption{Order of convergence of Strang, \CLTTwo, \CLTThree, and \CStrangThree\ applied to the 2D ADR problem
		\cref{eq:ADR}.}
	\label{fig:adr_convergence}
\end{figure}

\begin{figure}[htbp]
	\centering
	\includegraphics[width = \textwidth]{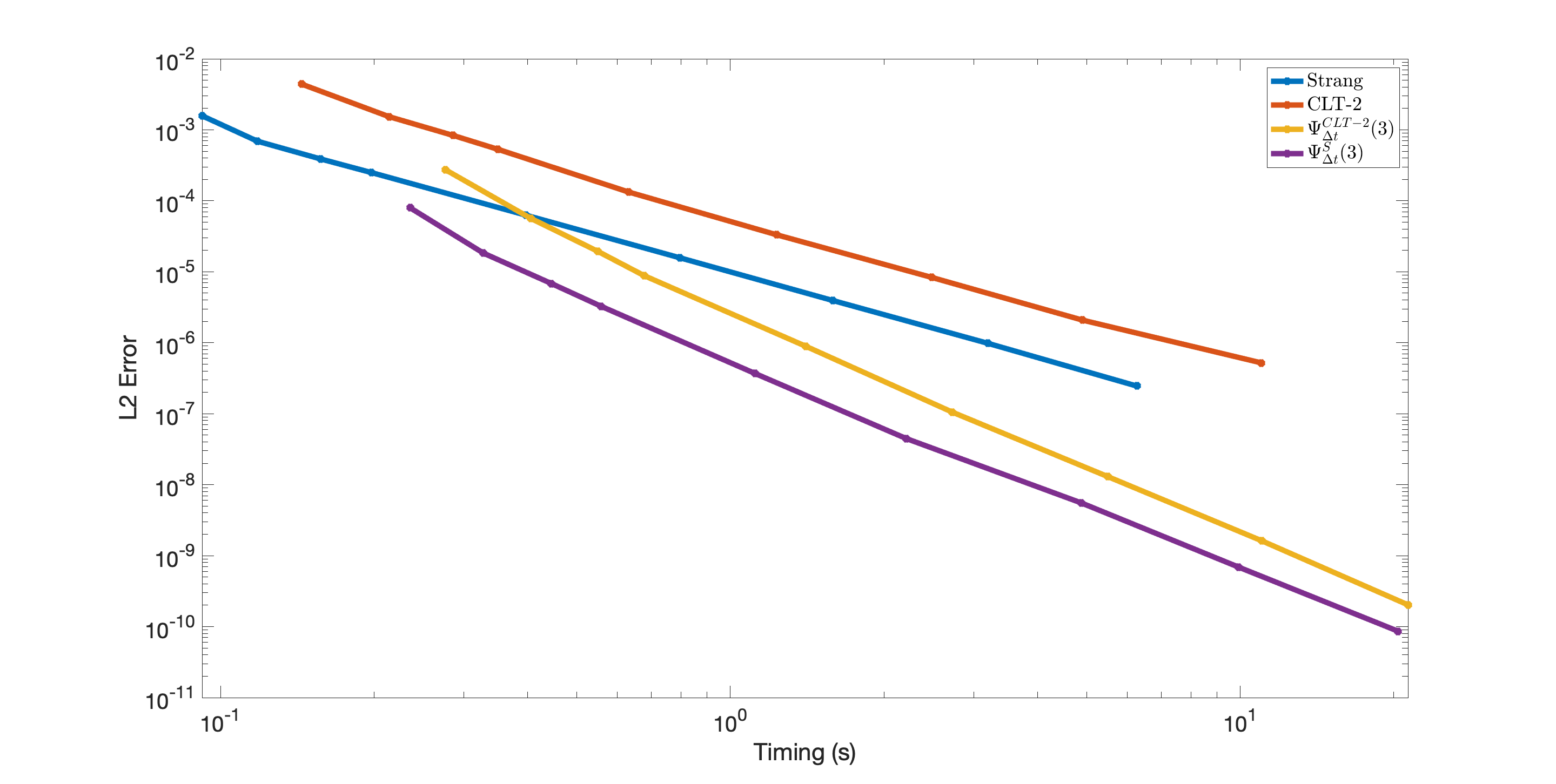}
	\caption{Work precision diagram of Strang, \CLTTwo, \CLTThree, and \CStrangThree\ applied to the 2D ADR problem
		\cref{eq:ADR}.}
	\label{fig:adr_wp}
\end{figure}

\subsection{$3$-split complex-valued ODE} 

In this section, we consider the complex-valued ODE 
\begin{equation}
	\label{eq:complex_ode}
	\dv{u}{t}  = iu + 0.05\,u - 0.5\,u^3
\end{equation}
with initial condition $u(0) = 0.1$ solved over the time interval
$[0,100]$. A complex-valued DE can be solved in two ways: as a
complex-valued DE or as a system of real-valued DEs of twice the
size. We examine the performance of real- and complex-valued
operator-splitting methods on complex differential equations. For this
example, we split \cref{eq:complex_ode} into $3$ operators.  To solve
\cref{eq:complex_ode} as a complex-valued ODE using $3$-split operator
splitting methods, we split \cref{eq:complex_ode} as follows:
\begin{align} \label{complex_ode_split} && \complexde^{[1]} (u^{[1]}) = iu^{[1]}, &&
	\complexde^{[2]}(u^{[2]}) = 0.05\, u^{[2]}, && \complexde^{[3]}(u^{[3]}) = -0.5\,(u^{[3]})^3.
\end{align}

To solve \cref{eq:complex_ode} as a system of real-valued differential
equations, let $u = x+iy$. Then \cref{eq:complex_ode} can be rewritten
as 
\begin{equation} \label{complex_ode_real}
	\left\{ 
	\begin{aligned}
		\dv{x}{t} & = - y+  0.05\,x  +0.15\,xy^2 -0.5\,x^3,  \\
		\dv{y}{t} & = \hspace*{2ex} x+0.05\,y -0.15\,x^2y + 0.5\,y^3. \\
	\end{aligned}  \right. 
\end{equation}
\Cref{complex_ode_real} can be split into three operators,
\begin{equation}
	\label{complex_ode_real_split} 
	\begin{aligned}
		&	\realde^{[1]}(x^{[1]},y^{[1]}) = \begin{bmatrix}
			-y^{[1]} \\
			\hspace*{1.5ex} x^{[1]} \\
		\end{bmatrix},  \\
		& \realde^{[2]}(x^{[2]},y^{[2]}) = \begin{bmatrix}
			0.05\,x^{[2]} \\ 
			0.05\,y^{[2]} \\
		\end{bmatrix},\\
		&
		\realde^{[3]}(x^{[3]},y^{[3]}) = \begin{bmatrix}
			\hspace*{2ex} 0.15\,x^{[3]}(y^{[3]})^2-0.5\,(x^{[3]})^3 \\
			-0.15\,(x^{[3]})^2y^{[3]} + 0.5\,(y^{[3]})^3
		\end{bmatrix}.
	\end{aligned}
\end{equation}

A reference solution of \cref{eq:complex_ode} is generated using
MATLAB {\tt ode45} with {\tt atol=1e-13} and {\tt rtol=1e-16}.  For
our experiments, we apply Strang, \CLTTwo, \CLTThree, \CStrangThree,
and \ThirdOrderReal~\cite{auzinger2016practical} to \cref{complex_ode_split} and
\cref{complex_ode_real_split}. The sub-systems are all solved using
Kutta's third-order method~\cite{kutta1901}. The error of the
numerical solution is measured by the mixed root-mean-square (MRMS)
error at $M$ space-time points defined by
\begin{equation*}
	\text{error} = [\text{MRMS}]_X = \sqrt{\frac{1}{M}
		\sum\limits_{i=1}^M \left(\frac{X_i^{\text{ref}}-X_i}{1+\abs{X_i^{\text{ref}}}} \right)^2},
\end{equation*}
where $X_i^{\text{ref}}$ and $X_i$, respectively, denote the reference
solution and the numerical solution at $100$ equally spaced temporal
points $x_n = n$, $n=1,2,\dots,100$.  
All methods exhibit the expected order of
convergence as seen in \cref{fig:complex_ode_convergence}.

\begin{figure}[htbp]
	\centering
	\includegraphics[width = \textwidth]{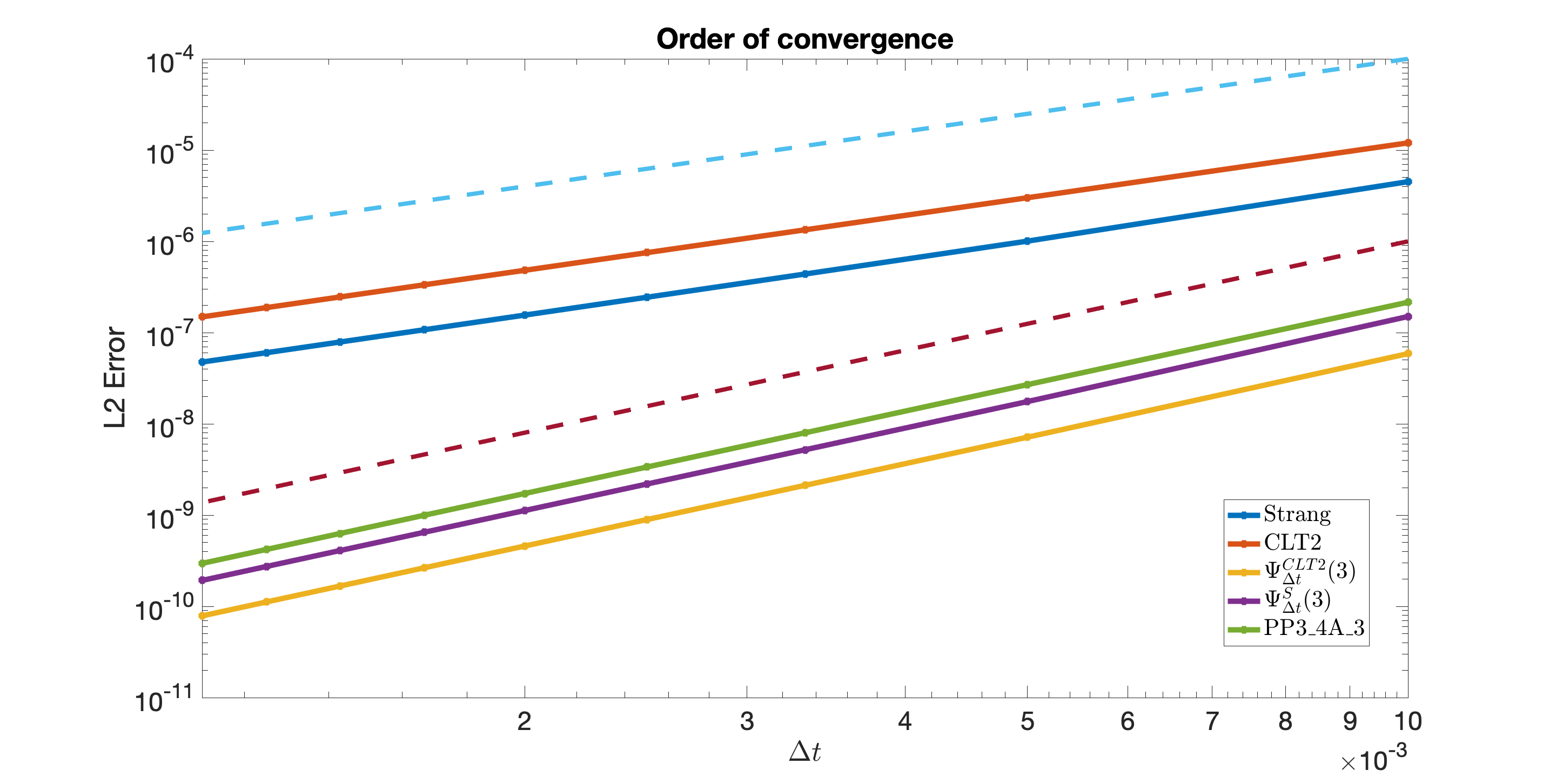}
	\caption{Convergence of Strang, \CLTTwo, \CLTThree, \CStrangThree\
		and \ThirdOrderReal~applied to the complex ODE
		\cref{eq:complex_ode}.}
	\label{fig:complex_ode_convergence}
\end{figure}

We also present the work-precision diagram of solving
\cref{eq:complex_ode} as a complex ODE and as a system of real ODEs
\cref{complex_ode_real}. We note that, in all cases, it is more
efficient to solve the system directly as a complex-valued ODE. For
relaxed tolerances, \CLTTwo~can be more efficient than Strang,
and \CLTThree\ can be more efficient than \CStrangThree. For tighter
tolerances, \CLTThree\ and \CStrangThree\ are comparable in terms of
efficiency.

\begin{figure}[htbp]
	\centering
	\includegraphics[width = \textwidth]{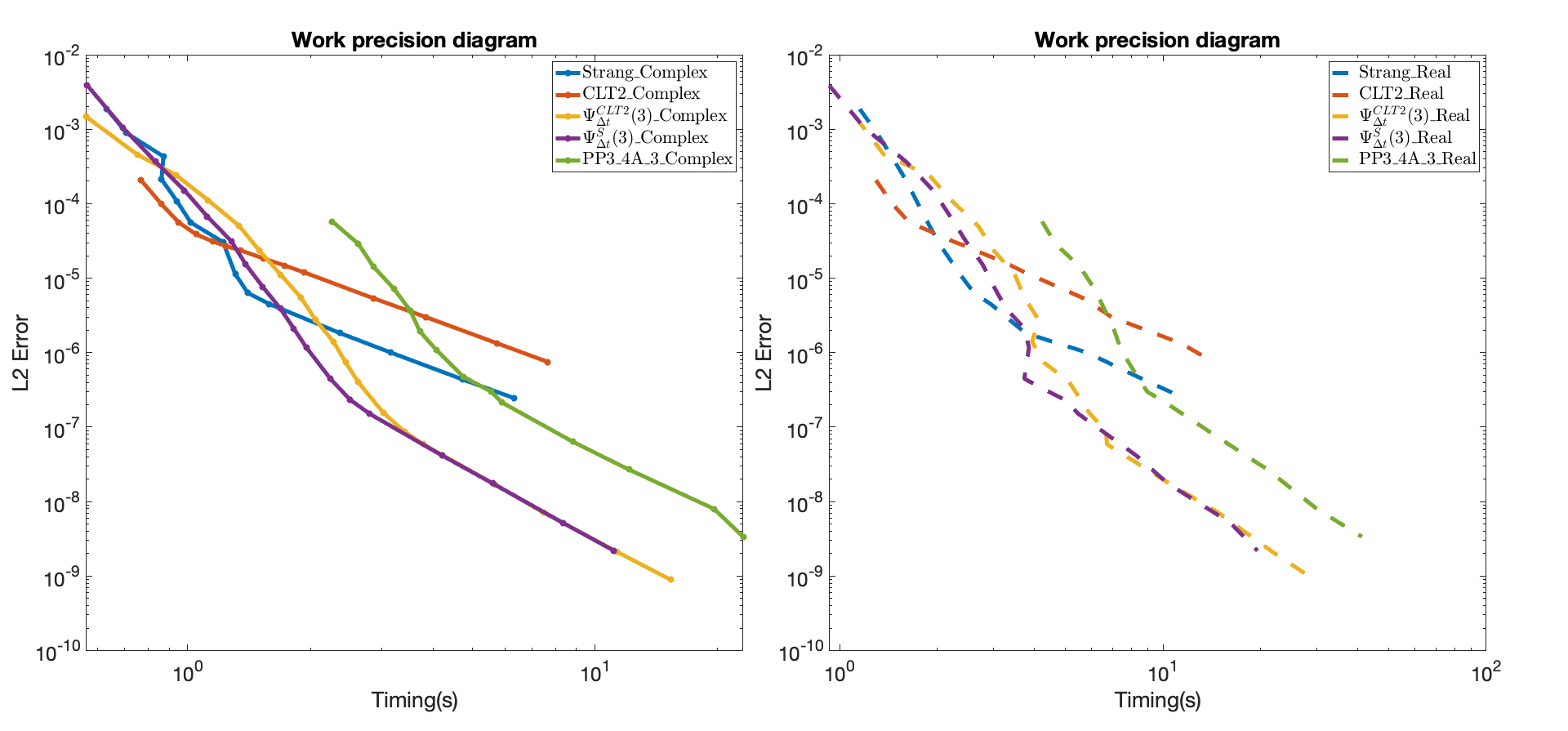}
	\caption{Work-precision diagram of Strang, \CLTTwo, \CLTThree,
		\CStrangThree, and \ThirdOrderReal\ applied to the complex ODE
		\cref{eq:complex_ode} and the system of real ODEs
		\cref{complex_ode_real}.}
	\label{fig:complex_ode_work_precision}
\end{figure}


\section{Conclusions}
\label{sec:complexLT_conclude}

In this paper, we presented the order conditions up to second order
for general $\Nop$-split operator-splitting methods and derived a pair
of computationally efficient two-stage, second-order, complex-valued
operator-splitting methods for arbitrary $\Nop$. We have called these
methods \emph{complex Lie--Trotter} (\CLTTwo) methods, and they join
the only other two general $\Nop$-split operator-splitting methods
known, namely, the Lie--Trotter and Strang methods. The \CLTTwo\
methods can also be used as base methods to derive efficient
high-order $\Nop$-split operator-splitting methods using
composition. We compared the performance of the complex-valued
operator-splitting methods against real-valued operator-splitting
methods of the same order on both real and complex ODEs. We
demonstrate that the third-order complex-valued operator-splitting
methods with positive real parts outperform third-order real-valued
operator-splitting methods. However, due to the increased
computational cost of complex arithmetic relative to real arithmetic
on standard present-day hardware, second-order complex-valued
operator-splitting methods may often be less efficient than the Strang
method, which only uses real arithmetic. For complex-valued ODEs, on
the other hand, complex-valued operator-splitting methods can be more
efficient than real-valued methods.  


\section*{Acknowledgments}
Support for this work was provided by the Natural Sciences and
Engineering Research Council of Canada under its Discovery Grant
program (RGPN-2020-04467).

\bibliographystyle{siamplain}
\bibliography{reference}
\end{document}